\documentclass[10pt]{amsart}

\usepackage[latin1]{inputenc}
\usepackage{amsmath}
\usepackage{amssymb}
\usepackage{amsthm}

\usepackage{multirow}
\usepackage{bigdelim}

\allowdisplaybreaks[3]

\theoremstyle{plain}
\newtheorem{theorem}{Theorem}[section]
\newtheorem{lemma}[theorem]{Lemma}

\theoremstyle{definition}
\newtheorem{definition}[theorem]{Definition}

\numberwithin{equation}{section}

\newcommand{\dosfilas}[2]{
  \ldelim[{2}{3mm} & #1 & \rdelim]{2}{3mm} \\%
  & #2 & & & &
}

\newcommand\F{{\mathcal F}}
\newcommand\HH{{\mathcal H}}
\newcommand\Fuku{\mathcal{F}_{1,\{k_1\}}}
\newcommand\Fdkd{\mathcal{F}_{2,\{k_2\}}}
\newcommand\tFdkd{\widetilde{\mathcal{F}}_{2,\{k_2\}}}

\newcommand\CC{{\mathbb C}}
\newcommand\RR{{\mathbb R}}
\newcommand\ZZ{{\mathbb Z}}
\newcommand\NN{{\mathbb N}}

\newcommand\Id{\operatorname{Id}}

\begin{document}

\title{Admissibility condition for exceptional Laguerre polynomials}
\thanks
{Partially supported by MTM2012-36732-C03-03 and
MTM2012-36732-C03-02 (Ministerio de Econom\'ia y Competitividad),
FQM-262, FQM-4643, FQM-7276 (Junta de Andaluc\'ia) and Feder Funds (European
Union).}

\author{Antonio J. Dur\'an}
\address{Departamento de An\'alisis Matem\'atico, Universidad de Sevilla, 41080 Sevilla, Spain}
\email{duran@us.es}

\author{Mario P\'erez}
\address{Departamento de Matem\'aticas and IUMA, Universidad de Zaragoza, 50009 Zaragoza, Spain}
\email{mperez@unizar.es}

\begin{abstract}
We prove a necessary and sufficient condition for the integrability of the weight associated to the exceptional Laguerre polynomials. This condition is very much related to the fact that the associated second order differential operator has no singularities in $(0,+\infty)$. 
\end{abstract}


\maketitle

\section{Introduction}
Exceptional orthogonal polynomials $p_n$, $n\in X\varsubsetneq \NN$, are complete orthogonal polynomial systems with respect to a positive measure which in addition
are eigenfunctions of a second order differential operator. They extend the classical families of Hermite, Laguerre and Jacobi. The most apparent difference between classical  orthogonal polynomials and their exceptional counterparts
is that the exceptional families have gaps in their degrees, in the
sense that not all degrees are present in the sequence of polynomials (as it happens with the classical families), although they form a complete orthonormal set of the underlying $L^2$ space defined by the orthogonalizing positive measure. This
means in particular that they are not covered by the hypotheses of Bochner's classification theorem (see \cite{B}) for classical orthogonal polynomials. The last few years have seen a great deal of activity in the area of exceptional orthogonal polynomials; see, for instance,
\cite{duch}, \cite{ducl}, \cite{GUKM1}, \cite{GUKM2} (where the adjective \textrm{exceptional} for this topic was introduced), \cite{GUKM3}, \cite{GUGM}, \cite{G}, \cite{GQ}, \cite{OS0}, \cite{OS3}, \cite{Qu}, \cite{STZ}, \cite{Ta} and the references therein.

One of the most interesting questions regarding exceptional polynomials is that of finding necessary and sufficient conditions so that the corresponding second order differential operator has no singularities in its domain (that is, it is regular). This question is very much related to the integrability of the weight with respect to which the exceptional polynomials are orthogonal. The purpose of this paper is to provide a complete answer to this question for exceptional Laguerre polynomials.

Exceptional Laguerre polynomials can be constructed using Wronskian type determinants whose entries are Laguerre polynomials. For a complex number $\alpha\in \hat \CC=\CC\setminus \{-1,-2,\dots \}$ we consider the Laguerre polynomials
\begin{equation}\label{deflap}
   L_n^{\alpha}(x) = \sum_{j=0}^n\frac{(-x)^j}{j!}\binom{n+\alpha}{n-j}
\end{equation}
(this and the next formulas can be found in \cite[vol.~II, pp.~188--192]{EMOT}; see also \cite[pp.~241-244]{KLS}).
They are always eigenfunctions of the second order differential operator
\[
D_{\alpha}=x\left(\frac{d}{dx}\right)^2+(\alpha+1-x)\frac{d}{dx},\quad D_{\alpha}(L_n^{\alpha})=-nL_n^{\alpha},\quad n\geq0.
\]
When $\alpha$ is real and $\alpha>-1$ they are also orthogonal with respect to the positive weight $x^\alpha e^{x}$, $x\in (0,+\infty)$. Otherwise, they are orthogonal with respect to a (signed) weight supported in a complex path.

As far as the authors know, the most general construction of exceptional Laguerre polynomials is given in \cite{ducl} (see also \cite{G}, \cite{GQ}, \cite{OS3}), and it proceeds as follows.

Denote by $\F=(F_1,F_2)$ a pair of finite sets of
positive integers, write $k_i$ for the number of elements of $F_i$,
$i=1,2$, and let $k=k_1+k_2$. The components of $\F$ can be the empty set.
We define the nonnegative integer $u_\F$ and the infinite set of nonnegative integers $\sigma_\F$, respectively, by
\begin{align*}
u_\F&=\sum_{f\in F_1}f+\sum_{f\in F_2}f-\binom{k_1+1}{2}-\binom{k_2}{2},
\\
\sigma_\F&=\{u_\F,u_\F+1,u_\F+2,\dots \}\setminus \{u_\F+f,f\in F_1\}.
\end{align*}
Notice that $\sigma_\F$ is formed by the nonnegative integers of the form $n+u_\F$ with $n\not \in F_1$. The infinite set $\sigma_\F$ will be the set of indices for the exceptional Laguerre polynomials associated to $\F$.

For each pair $\F=(F_1,F_2)$ of finite sets of positive integers and a complex number $\alpha\in \hat \CC$, we define the polynomials
\begin{equation}\label{deflaxi}
L_n^{\alpha ;\F}(x)= \left|
  \begin{array}{@{}c@{}cccc@{}c@{}}
    & L_{n-u_\F}^{\alpha}(x)&(L_{n-u_\F}^{\alpha})'(x)&\cdots &(L_{n-u_\F}^{\alpha})^{(k)}(x) & \\
    \dosfilas{ L_{f}^{\alpha}(x) & (L_{f}^{\alpha})'(x) &\cdots & (L_{f}^{\alpha})^{(k)}(x) }{f\in F_1} \\
    \dosfilas{ L_{f}^{\alpha}(-x) & L_{f}^{\alpha+1}(-x) & \cdots & L_{f}^{\alpha +k}(-x) }{f\in F_2}
  \end{array}
  \right|,
\end{equation}
for $n\in \sigma_\F$. Along this paper, we use the following notation:
given a finite set of positive integers $F=\{f_1,\ldots , f_m\}$, the expression
\[
  \begin{bmatrix}
   z_{f,1} & z_{f,2} &\cdots & z_{f,m}\\
   f\in F
  \end{bmatrix}
\]
inside a matrix or a determinant will mean the submatrix defined by
\[
   \begin{pmatrix}
   z_{f_1,1} & z_{f_1,2} &\cdots & z_{f_1,m}\\
   \vdots &\vdots &\ddots &\vdots \\
   z_{f_m,1} & z_{f_m,2} &\cdots & z_{f_m,m}
   \end{pmatrix} .
\]
The determinant~\eqref{deflaxi} should be understood in this form.

Notice that if both components of $\F$ are the empty set, we get $\sigma _\F=\NN$ and the exceptional Laguerre polynomials $L_n^{\alpha ;\F}$, $n\in\sigma_\F$, reduce to the Laguerre polynomials.

Write $\Omega_{\F}^{\alpha}(x)$ for the polynomial defined by
\[
\Omega_{\F}^{\alpha}(x)=\left|
  \begin{array}{@{}c@{}cccc@{}c@{}}
    \dosfilas{ L_{f}^{\alpha}(x) & (L_{f}^{\alpha})'(x) &\cdots & (L_{f}^{\alpha})^{(k-1)}(x) }{f\in F_1} \\
    \dosfilas{ L_{f}^{\alpha}(-x) & L_{f}^{\alpha+1}(-x) & \cdots & L_{f}^{\alpha +k-1}(-x) }{f\in F_2}
  \end{array}
  \right| .
\]
In \cite{ducl}, one of us has proved (see Theorem 5.2) that the polynomials $(L_n^{\alpha ;\F})_{n\in\sigma_\F}$ are eigenfunctions of the following second order differential operator:
\begin{equation}\label{sodolax}
D_{\F}=x\partial^2+h_1(x)\partial+h_0(x),
\end{equation}
where $\partial=d/dx$ and
\begin{align*}
h_1(x)&=\alpha +k+1-x-2x\frac{(\Omega_\F^\alpha)'(x)}{\Omega_\F^\alpha(x)},
\\
h_0(x)&=-k_1-u_\F +(x-\alpha -k)\frac{(\Omega_\F^\alpha)'(x)}{\Omega_\F^\alpha(x)}+x\frac{(\Omega_\F^\alpha)''(x)}{\Omega_\F^\alpha(x)}.
\end{align*}
More precisely $D_\F(L_n^{\alpha;\F})=-nL_n^{\alpha;\F}(x)$. Actually, only real numbers $\alpha$ were considered in \cite{ducl}, but since $D_\F(L_n^{\alpha;\F})$ and $nL_n^{\alpha;\F}(x)$ are analytic functions in $\hat \CC$, the result holds also in $\hat \CC$.

Exceptional Laguerre polynomials are formally orthogonal with respect to the weight
\begin{equation}\label{exlw}
\omega_{\alpha;F} =\frac{x^{\alpha +k}e^{-x}}{\Omega_{\F}^{\alpha}(x)^2},\quad x>0 .
\end{equation}
Hence, for these exceptional Laguerre polynomials the regularity of the associated second order differential operator in $(0,+\infty)$, and the existence of a positive measure with respect to which they are orthogonal are related to the fact that $\Omega_\F^\alpha (x)\not =0$, $x\ge 0$.
This problem gave rise in \cite{ducl} to the concept of admissibility for a real number $c$ and the pair $\F$.

\begin{definition} Let $\F=(F_1,F_2)$ be a pair of finite sets of positive integers. For a real number $c\not=0,-1,-2,\dots$, write
$\hat c=\max \{-[c],0\}$, where $[c]$ denotes the integer part of $c$ (i.e. $[c]=\max\{s\in \ZZ: s\le c\}$). We say that $c$ and $\F$ are admissible if for all $n\in \NN $
\begin{equation}\label{defadmi}
\frac{\prod_{f\in F_1}(n-f)\prod_{f\in F_2}(n+c+f)}{(n+c)_{\hat c}}\ge 0.
\end{equation}
\end{definition}
As usual $(a)_j$ will denote the Pochhammer symbol defined by
\[
(a)_0=1,\quad \quad (a)_j=a(a+1)\dots (a+j-1),\quad \mbox{for $j\ge 1$, $a\in \CC$}.
\]
When $\alpha +1$ and $\F$ are admissible, it was proved in \cite{ducl} that $\alpha +k>-1$ and the determinant $\Omega_F^\alpha $
does not vanish in $[0,+\infty)$ (and hence, the weight~\eqref{exlw} is integrable in $(0,+\infty)$). It was also conjectured that the converse is also true.

The purpose of this paper is to prove this conjecture:

\begin{theorem}\label{t1i} Let $\alpha$ be a real number with $\alpha \not =-1,-2,\dots$ If $\alpha +k>-1$ and $\Omega_{\F}^{\alpha}(x)\not =0$, $x\geq 0$, then $\alpha+1$ and $\F$ are admissible.
\end{theorem}

This theorem will be an easy consequence of the following complex orthogonality for the exceptional Laguerre polynomials (both will be proved in Section~\ref{proofs}). For $r>0$, we denote by $\Lambda_r$ the positively oriented complex path formed by the half lines
$x\pm r i$, $x\in [0,+\infty)$, and the left hand side of the circle with center at $0$ and radius $r$ (joining the two half lines). We also consider the branch of the logarithmic function $\log z$ defined in $\CC\setminus [0,+\infty)$ with $\log i=i \pi/2$. Given a complex number $a$ we then define $z^a=e^{a\log z}$, which is an analytic function of $z$ in $\CC\setminus [0,+\infty)$.

\begin{lemma}\label{l1i} Given $\alpha\in \hat \CC$ and a pair $\F=(F_1,F_2)$ of finite sets of positive integers, there exists $r>0$ such that
\begin{multline*}
   \int_{\Lambda_r} L_{n+u_\F}^{\alpha;\F}(z)L_{m+u_\F}^{\alpha;\F}(z)
   \frac{z^{\alpha +k}e^{-z}}{\Omega_\F^{\alpha}(z)^2} \, dz
   \\
   =(e^{2\pi \alpha i}-1)
   \frac{\Gamma(n+\alpha+1)\prod_{f\in F_1}(n -f)
   \prod_{f\in F_2}(n+\alpha+f+1)}
   {n!}\delta_{n,m},
\end{multline*}
for every $n,m\not \in F_1$.
\end{lemma}

We finish this paper with an appendix where the admissibility condition~\eqref{defadmi} is rewritten so that it allows an easy generation of examples of admissible real numbers $c$ and pairs $\F$ of finite sets of positive integers.

\section{Preliminaries}
From now on, $\F=(F_1,F_2)$ will denote a pair of finite sets of
positive integers. We will write $F_1=\{ f_1^{1\rceil},\dots ,
f_{k_1}^{1\rceil}\}$, $F_2=\{ f_1^{2\rceil},\dots , f_{k_2}^{2\rceil}\}$, with
$f_i^{j\rceil}<f_{i+1}^{j\rceil}$.

Hence $k_j$ is the number of elements of $F_j$,
$j=1,2$. We write $k=k_1+k_2$. The components of $\F$ can be the empty set.

For a pair $\F=(F_1,F_2)$ of positive integers we denote by
$\F_{j,\{ i\}}$, $i=1,\ldots , k_j$, $j=1,2$, the pair of finite sets of positive integers
defined by
\begin{align}
\F_{1,\{ i\}}&=(F_1\setminus \{f_i^{1\rceil}\},F_2),
\notag\\\label{deff2}
\F_{2,\{ i\}}&=(F_1,F_2\setminus \{f_i^{2\rceil}\}).
\end{align}

Darboux transformations are an important tool for constructing exceptional orthogonal polynomials.

\begin{definition}
Given a system $(T,(\phi_n)_n)$ formed by a second order differential operator $T$ and a sequence $(\phi_n)_n$ of eigenfunctions for $T$, $T(\phi_n)=\pi_n\phi_n$, by a Darboux transform of the system $(T,(\phi_n)_n)$ we
mean the following. For a real number $\lambda$, we factorize $T-\lambda \Id$ as the product of two first order differential operators $T=BA+\lambda \Id$ ($\Id$ denotes the identity operator). We then produce a new system consisting of the operator $\hat T$, obtained by reversing the order of the factors,
$\hat T = AB+\lambda \Id$, and the sequence of eigenfunctions $\hat \phi_n =A(\phi_n)$: $\hat T(\hat \phi_n)=\pi_n\hat\phi_n$.
We say that the system $(\hat T,(\hat\phi_n)_n)$ has been obtained by applying a Darboux transformation with parameter $\lambda$ to
 the system $(T,(\phi_n)_n)$.
\end{definition}

In \cite{ducl}, the second order differential operator $D_\F$~\eqref{sodolax} for the exceptional Laguerre polynomials was factorized as a product of two first order differential operators (see Lemma 5.4). As explained there, such a factorization can be done by choosing one of the components of $\F=(F_1,F_2)$ and removing one element in the chosen component (an iteration shows then that the system $(D_\F, (L_n^{\alpha; \F})_{n\in \sigma_\F})$ can be constructed by applying a sequence of $k$ Darboux transforms to the Laguerre system). This factorization will be the key to prove Lemma \ref{l1i} in the introduction of this paper.

\begin{lemma}[Lemma 5.4 of \cite{ducl}]\label{lfel} Let $\F=(F_1,F_2)$ be a pair of finite sets of positive integers.
\begin{enumerate}
\item[a)] If $F_1\not =\emptyset$, we define the first order differential operators $A_{1}^{\F}$, $B_{1}^{\F}$ as
\begin{align}
   \label{opea1}
   A_1^\F &=-\frac{\Omega_\F^\alpha(x)}{\Omega_{\Fuku}^\alpha(x)}\partial
   +\frac{(\Omega_\F^\alpha)'(x)}{\Omega_{\Fuku(x)}^\alpha},
   \\
   \label{opeb1}
   B_1^\F &=\frac{-x\Omega_{\Fuku}^\alpha(x)}{\Omega_{\F}^\alpha(x)}\partial
   +\frac{x(\Omega_{\Fuku}^\alpha)'(x)+(x-\alpha-k)\Omega_{\Fuku}^\alpha(x)}
   {\Omega_{\F}^\alpha(x)},
\end{align}
where $k_1$ is the number of elements of $F_1$ and $\Fuku$ is defined by~\eqref{deff2}. Then for $n\not \in F_1$
\begin{align*}
   A_1^\F(L_{n+u_{\Fuku}}^{\alpha; \Fuku})(x)
   &= L_{n+u_\F}^{\alpha ;\F}(x),
   \\
   B_1^\F(L_{n+u_\F}^{\alpha ;\F})(x)
   &= -(n-f_{k_1}^{1\rceil}) L_{n+u_{\Fuku}}^{\alpha; \Fuku}(x).
\end{align*}
Moreover,
\begin{align*}
   D_{\Fuku} &=B_1^\F A_1^\F-(f_{k_1}^{1\rceil}+u_{\Fuku})\Id,
   \\
   D_{\F} &=A_1^\F B_1^\F-(f_{k_1}^{1\rceil}+u_\F)\Id.
\end{align*}
\item[b)] If $F_2\not =\emptyset$, we define the first order differential operators $A_{2}^{\F}$, $B_{2}^{\F}$ as
\begin{align}\label{opea2}
   A_2^\F &=-\frac{\Omega_\F^\alpha(x)}{\Omega_{\Fdkd}^\alpha(x)}\partial
   +\frac{(\Omega_\F^\alpha)'(x)+\Omega_\F^\alpha(x)}{\Omega_{\Fdkd}^\alpha(x)},
   \\
   \label{opeb2}
   B_2^\F &=\frac{-x\Omega_{\Fdkd}^\alpha(x)}{\Omega_{\F}^\alpha(x)}\partial
   +\frac{x(\Omega_{\Fdkd}^\alpha)'(x)-(\alpha+k)\Omega_{\Fdkd}^\alpha(x)}
   {\Omega_{\F}^\alpha(x)},
\end{align}
where $k_2$ is the number of elements of $F_2$ and $\Fdkd$ is defined by~\eqref{deff2}. Then for $n\not \in F_1$
\begin{align*}
   A_2^\F(L_{n+u_{\Fdkd}}^{\alpha; \Fdkd})(x)
   &= L_{n+u_\F}^{\alpha ;\F}(x),
   \\
   B_2^\F(L_{n+u_\F}^{\alpha ;\F})(x)
   &= -(\alpha+n+f_{k_2}^{2\rceil}+1) L_{n+u_{\Fdkd}}^{\alpha; \Fdkd}(x).
\end{align*}
Moreover,
\begin{align*}
   D_{\Fdkd}
   &=B_2^\F A_2^\F+(\alpha+f_{k_2}^{2\rceil}-u_{\Fdkd}+1) \Id,
   \\
   D_{\F} &=A_2^\F B_2^\F+(\alpha +f_{k_2}^{2\rceil}-u_\F+1) \Id.
\end{align*}
\end{enumerate}
\end{lemma}

In \cite{ducl}, only the case of $F_2$ is considered (i.e., the operators~\eqref{opea2} and~\eqref{opeb2}), but the result for $F_1$ can be obtained in a completely similar way.

\section{Proofs of Theorem \ref{t1i} and Lemma \ref{l1i}}
\label{proofs}
Both results are a consequence of the following three lemmas. Let $P$ and $Q$ be two polynomials. We associate to them the following four first order differential operators:
\begin{align}\label{ope1}
A_1&=\frac{-P}{Q}\frac{d}{dx}+\frac{P'}{Q},\quad &B_1&=\frac{-xQ}{P}\frac{d}{dx}+\frac{xQ'+(x-\alpha) Q}{P},\\\label{ope2}
A_2&=\frac{-P}{Q}\frac{d}{dx}+\frac{P'+P}{Q},\quad &B_2&=\frac{-xQ}{P}\frac{d}{dx}+\frac{xQ'-\alpha Q}{P}.
\end{align}

\begin{lemma}\label{lm1} Assume that $\alpha \in \hat \CC$ and the polynomials $P$ and $Q$ do not vanish in $\Lambda_r$. If $p$ and $q$ are polynomials, then
\[
\int_{\Lambda_r} p(z)(A_i(q(z))\frac{z^\alpha e^{-z}}{P(z)^2}\, dz=
-\int_{\Lambda_r} B_i(p(z))q(z)\frac{z^{\alpha-1} e^{-z}}{Q(z)^2}\, dz,\quad i=1,2.
\]
\end{lemma}

\begin{proof}
Integrating by parts gives
\begin{multline*}
   \int_{\Lambda_r} - \frac{p(z) z^\alpha e^{-z}}{P(z) Q(z)} q'(z) \, dz
   \\
   = \left[ - \frac{p(z) z^\alpha e^{-z}}{P(z) Q(z)} q(z)
   \right]_{z\to +\infty - i r}^{z \to +\infty+ir}
   + \int_{\Lambda_r}
   \left(\frac{p(z) z^\alpha e^{-z}}{P(z) Q(z)}\right)' q(z) \, dz,
\end{multline*}
which reduces to
\[
   \int_{\Lambda_r}
   \left(\frac{p(z) z^\alpha e^{-z}}{P(z) Q(z)}\right)' q(z) \, dz,
\]
due to the factor $e^{-z}$. Differentiating this quotient like a product rather than like a quotient yields
\begin{align*}
   \int_{\Lambda_r} -\frac{p(z) z^\alpha e^{-z}}{P(z) Q(z)} q'(z) \, dz
   &= \int_{\Lambda_r} \frac{p'(z) z^\alpha e^{-z}}{P(z) Q(z)} q(z) \, dz
   + \int_{\Lambda_r}
   \frac{p(z) \alpha z^{\alpha-1} e^{-z}}{P(z) Q(z)} q(z) \, dz
   \\
   &\qquad- \int_{\Lambda_r}
   \frac{p(z) z^\alpha e^{-z} (P(z) + P'(z))}{P(z)^2 Q(z)} q(z) \, dz
   \\
   &\qquad- \int_{\Lambda_r} \frac{p(z) z^\alpha e^{-z} Q'(z)}{P(z)Q(z)^2} q(z) \, dz,
\end{align*}
that is,
\begin{multline*}
   \int_{\Lambda_r} p(z) \left(
   - \frac{P(z) q'(z)}{Q(z)} + \frac{P(z) + P'(z)}{Q(z)} q(z)
   \right) \frac{z^\alpha e^{-z}}{P(z)^2} \, dz
   \\
   = \int_{\Lambda_r}
   \left(
   \frac{z Q(z) p'(z)}{P(z)} + \frac{\alpha p(z) Q(z)}{P(z)}
   - \frac{p(z) z Q'(z)}{P(z)}
   \right)
    q(z) \frac{z^{\alpha-1}e^{-z}}{Q(z)^2} \, dz.
\end{multline*}
This proves the lemma for $i=2$. The case $i=1$ follows easily from the equalities $A_1 = A_2 - \frac{P}{Q}$ and $B_1 = B_2 + \frac{x Q}{P}$.
\end{proof}

\begin{lemma}\label{lm2} Assume that $\alpha$ is a real number with $\alpha>-1$. Assume also that the polynomial $P$ does not vanish neither in $\Lambda_r$ nor in its interior. If $p$ is a polynomial, then
\[
   \int_{\Lambda_r} p(z)\frac{z^\alpha e^{-z}}{P(z)^2} \, dz
   = (e^{2\pi \alpha i}-1)
   \int_0^\infty p(x)\frac{x^\alpha e^{-x}}{P(x)^2}\, dx.
\]
\end{lemma}

\begin{proof}
Let us cut the path $\Lambda_r$ this way: fix some $R > 0$ and consider the path $\Lambda_{r,R} = \Lambda_r \cap \{\Re z \leq R\}$, positively oriented like $\Lambda_r$.
Let us now take $0 < \varepsilon < r$ and consider the horizontal segments
\begin{align*}
   [i\varepsilon, R + i\varepsilon],
   \\
   [R-i\varepsilon, - i\varepsilon],
\end{align*}
and the vertical segments
\begin{align*}
   [R-ir, R - i\varepsilon],
   \\
   [R + i \varepsilon, R + ir],
\end{align*}
where each segment $[z,w]$ is taken from $z$ to $w$. Finally, consider the semicircle $\gamma_\varepsilon = \{|z|=\varepsilon,\ \Re z \le 0\}$, negatively oriented. Then, the path
\[
   \Gamma_{r,R,\varepsilon} = \Lambda_{r,R} \cup [R-ir, R - i\varepsilon]
   \cup [R-i\varepsilon, - i\varepsilon]
   \cup \gamma_\varepsilon
   \cup [i\varepsilon, R + i\varepsilon]
   \cup [R + i \varepsilon, R + ir]
\]
is closed and the function $p(z)\frac{z^\alpha e^{-z}}{P^2(z)}$ is holomorphic in  an open, simply connected set containing the path and its interior. By Cauchy's theorem,
\[
   \int_{\Gamma_{r,R,\varepsilon}} p(z)\frac{z^\alpha e^{-z}}{P(z)^2} \, dz = 0.
\]
Now let us take $\varepsilon \to 0^+$ and analyse each path separately. On the semicircle $\gamma_\varepsilon$ we have
\[
   \operatorname{length}(\gamma_\varepsilon) \sup_{z \in \gamma_\varepsilon}
   \left| p(z)\frac{z^\alpha e^{-z}}{P(z)^2} \right|
   \leq C \varepsilon^{1+\alpha},
\]
for some constant $C$ not depending on $\varepsilon$, so that
\[
   \lim_{\varepsilon \to 0^+} \int_{\gamma_\varepsilon}
   p(z)\frac{z^\alpha e^{-z}}{P(z)^2} \, dz = 0.
\]
The limit on the two vertical segments reduces to the integral on the vertical segment $[R-ir, R+ir]$. The limits on the intervals $[i\varepsilon, R + i\varepsilon]$ and $[R-i\varepsilon, - i\varepsilon]$ are
\[
   \lim_{\varepsilon \to 0^+}
   \int_{[i\varepsilon, R + i\varepsilon]}
   p(z)\frac{z^\alpha e^{-z}}{P(z)^2} \, dz
   =\int_0^R p(x)\frac{x^\alpha e^{-x}}{P(x)^2} \, dx
\]
and
\[
   \lim_{\varepsilon \to 0^+}
   \int_{[R-i\varepsilon, - i\varepsilon]}
   p(z)\frac{z^\alpha e^{-z}}{P(z)^2} \, dz
   = - e^{2\pi \alpha i}
   \int_0^R p(x)\frac{x^\alpha e^{-x}}{P(x)^2} \, dx,
\]
due to the branch of the logarithm we have chosen. In both cases, the dominate convergence theorem applies. This proves that
\[
   \int_{\Lambda_{r,R}} p(z)\frac{z^\alpha e^{-z}}{P(z)^2} \, dz
   = (e^{2\pi \alpha i}-1) \int_0^R p(x)\frac{x^\alpha e^{-x}}{P(x)^2} \, dx
   - \int_{[R-ir, R+ir]} p(z)\frac{z^\alpha e^{-z}}{P(z)^2} \, dz.
\]
Taking limit as $R \to +\infty$ proves the lemma, since
\[
   \lim_{R \to \infty}
   R \sup_{z \in [R-ir, R+ir]} \left| p(z)\frac{z^\alpha e^{-z}}{P(z)^2} \right|
   = 0.
   \qedhere
\]
\end{proof}

\begin{lemma}\label{lm3} If $\alpha \in \hat \CC$, then
\begin{equation}
\label{eqlm3}
   \int_{\Lambda_r} L_n^\alpha(z)L_m^\alpha(z)z^\alpha e^{-z}\, dz
   =(e^{2\pi \alpha i}-1)\frac{\Gamma(n+\alpha +1)}{n!}\delta_{n,m}
\end{equation}
for every nonnegative integers $m,n$.
\end{lemma}

\begin{proof}
Let us fix two nonnegative integers $m,n$. It is rather elementary to deduce from~\eqref{deflap} that the left hand side in~\eqref{eqlm3} is
an entire function of $\alpha$, while the right hand side is analytic in $\alpha \in \hat{\CC}$ (though it can be extended to an entire function, as well).
Thus, proving~\eqref{eqlm3} for $\alpha \in (-1,+\infty)$ will be enough. Now, for $\alpha \in (-1,+\infty)$ this follows from Lemma~\ref{lm2} and the well-known classical orthogonality relation
\[
   \int_0^\infty L_n^\alpha(x)L_m^\alpha(x)x^\alpha e^{-x}\, dx
   = \frac{\Gamma(n+\alpha +1)}{n!}\delta_{n,m}.
   \qedhere
\]
\end{proof}

We are now ready to prove Lemma \ref{l1i} and Theorem \ref{t1i}.

\begin{proof}[Proof of Lemma \ref{l1i}]
Consider the family of finite sets of positive integers $\Psi=\{\HH =(H_1,H_2): H_1\subset F_1,H_2\subset F_2\}$. Since the number of elements of $\Psi$ is finite, we can choose a positive number $r>0$, such that the polynomial $\Omega_\HH^{\alpha}$ does not vanish in the complex path $\Lambda_r$ if $\HH\in \Psi$.

Assume first that $F_1\not=\emptyset$, and write $P=\Omega_\F^{\alpha}$ and $Q=\Omega_{\Fuku}^{\alpha}$.

It is then easy to check that the operators $A_1$ and $B_1$ defined by~\eqref{ope1}, with $\alpha+k$ instead of $\alpha$, coincide with the operators $A_1^\F$ and $B_1^\F$ defined by~\eqref{opea1} and~\eqref{opeb1}. By writing $p_n=L_{n+u_\F}^{\alpha ;\F}$ and $q_n=L_{n+u_{\Fuku}}^{\alpha; \Fuku}$, Lemma \ref{lfel} gives
\[
   A_1^\F(q_n) = p_n,
   \qquad
   B_1^\F(p_n) = -(n-f_{k_1}^{1\rceil}) q_n.
\]
Hence, using Lemma \ref{lm1}, we get
\begin{align*}
\int_{\Lambda_r} L_{n+u_\F}^{\alpha;\F}(z) &L_{m+u_\F}^{\alpha;\F}(z)
\frac{z^{\alpha +k}e^{-z}}{\Omega_\F^{\alpha}(z)^2} \, dz
=\int_{\Lambda_r} p_n(z)p_m(z)\frac{z^{\alpha +k} e^{-z}}{P(z)^2} \, dz
\\
&=\int_{\Lambda_r} p_n(z)A_1^\F(q_m)(z)\frac{z^{\alpha +k} e^{-z}}{P(z)^2} \, dz
\\
&=-\int_{\Lambda_r} B_1^\F(p_n)(z)q_m(z)\frac{z^{\alpha +k-1} e^{-z}}{Q(z)^2} \, dz
\\
&=(n-f_{k_1}^{1\rceil}) \int_{\Lambda_r} q_n(z)q_m(z)\frac{z^{\alpha +k-1} e^{-z}}{Q(z)^2} \, dz
\\
&=(n-f_{k_1}^{1\rceil}) \int_{\Lambda_r}L_{n+u_{\Fuku}}^{\alpha;\Fuku}(z)L_{m+u_{\Fuku}}^{\alpha;\Fuku}(z)\frac{z^{\alpha +k-1}e^{-z}}{\Omega_{\Fuku}^{\alpha}(z)^2} \, dz.
\end{align*}
Repeating the process, we have
\begin{multline*}
\int_{\Lambda_r}L_{n+u_\F}^{\alpha;\F}(z)L_{m+u_\F}^{\alpha;\F}(z)\frac{z^{\alpha +k}e^{-z}}{\Omega_F^{\alpha}(z)^2} \, dz
\\
=\prod_{f\in F_1}(n-f)\int_{\Lambda_r}L_{n+u_{\widetilde \F}}^{\alpha;\widetilde \F}(z)L_{m+u_{\widetilde \F}}^{\alpha;\widetilde \F}(z)\frac{z^{\alpha +k-k_1}e^{-z}}{\Omega_{\widetilde \F}^{\alpha}(z)^2}\, dz,
\end{multline*}
where $\widetilde \F=(\emptyset, F_2)$.

We now proceed in the same way with $F_2$. We hence write $P=\Omega_{\widetilde \F}^{\alpha}$ and $Q=\Omega_{\tFdkd}^{\alpha}$.
The operators $A_2$ and $B_2$ defined by~\eqref{ope2}, with $\alpha+k-k_1$ instead of $\alpha$, coincide with the operators $A_2^{\widetilde \F}$ and $B_2^{\widetilde \F}$ defined by~\eqref{opea2} and~\eqref{opeb2}. By writing $p_n=L_{n+u_{\widetilde \F}}^{\alpha ;\widetilde \F}$ and $q_n=L_{n+u_{\tFdkd}}^{\alpha; \tFdkd}$, Lemma \ref{lfel} gives
\[
   A_2^{\widetilde \F}(q_n) = p_n,
   \qquad
   B_2^{\widetilde \F}(p_n) = -(\alpha +1+n+f_{k_2}^{2\rceil}) q_n.
\]
Hence, using Lemma \ref{lm1}, we get
\begin{multline*}
\int_{\Lambda_r}L_{n+u_{\widetilde \F}}^{\alpha;\widetilde \F}(z)L_{m+u_{\widetilde \F}}^{\alpha;\widetilde \F}(z)\frac{z^{\alpha +k-k_1}e^{-z}}{\Omega_{\widetilde \F}^{\alpha}(z)^2} \, dz
\\
=(\alpha+1+n+f_{k_2}^{2\rceil})\int_{\Lambda_r} L_{n+u_{\tFdkd}}^{\alpha;\tFdkd}(z)L_{m+u_{\tFdkd}}^{\alpha;\tFdkd}(z)\frac{z^{\alpha +k-k_1-1}e^{-z}}{\Omega_{\tFdkd}^{\alpha}(z)^2} \, dz.
\end{multline*}
Since the exceptional Laguerre polynomials associated to the pair $(\emptyset,\emptyset)$ are the Laguerre polynomials, repeating the process we have
\begin{multline*}
\int_{\Lambda_r}L_{n+u_{\widetilde \F}}^{\alpha;\widetilde \F}(z)L_{m+u_{\widetilde \F}}^{\alpha;\widetilde \F}(z)\frac{z^{\alpha +k-k_1}e^{-z}}{\Omega_{\widetilde \F}^{\alpha}(z)^2} \, dz
\\
=\prod_{f\in F_2}(\alpha+1+n+f)\int_{\Lambda_r}L_{n}^{\alpha}(z)L_{m}^{\alpha}(z)z^{\alpha }e^{-z} \, dz.
\end{multline*}
Lemma \ref{lm3} finally gives
\begin{multline*}
   \int_{\Lambda_r} L_{n+u_\F}^{\alpha;\F}(z)L_{m+u_\F}^{\alpha;\F}(z)
   \frac{z^{\alpha +k}e^{-z}}{\Omega_F^{\alpha}(z)^2} \, dz
   \\
   = (e^{2\pi \alpha i}-1)
   \frac{\Gamma(n+\alpha+1) \prod_{f\in F_1}(n -f)
   \prod_{f\in F_2}(n+\alpha+f+1)}{n!} \delta_{n,m}.
   \qedhere
\end{multline*}
\end{proof}

\begin{proof}[Proof of Theorem \ref{t1i}]
Take $r$ as in the proof of Lemma \ref{l1i}. We then have, for $n\not \in F_1$,
\begin{multline*}
   \int_{\Lambda_r} L_{n+u_\F}^{\alpha;\F}(z)^2
   \frac{z^{\alpha +k}e^{-z}}{\Omega_{\F}^{\alpha}(z)^2} \, dz
   \\
   = (e^{2\pi \alpha i}-1)
   \frac{\Gamma(n+\alpha+1) \prod_{f\in F_1}(n -f)
   \prod_{f\in F_2}(n+\alpha+f+1)}{n!}.
\end{multline*}
On the other hand, since $\alpha+k>-1$ and $\Omega_{\F}^{\alpha}(x)\not =0$ for $x\ge 0$, we can also assume in the choice of $r$ that $\Omega_{\F}^{\alpha}$ does not vanish in the interior of $\Lambda_r$. Then, Lemma \ref{lm2} gives
\[
   \int_{\Lambda_r} L_{n+u_\F}^{\alpha;\F}(z)^2
   \frac{z^{\alpha +k}e^{-z}}{\Omega_{\F}^{\alpha}(z)^2} \, dz
   = (e^{2\pi \alpha i}-1)
   \int_0^\infty L_{n+u_\F}^{\alpha;\F}(x)^2
   \frac{x^{\alpha +k}e^{-x}}{\Omega_{\F}^{\alpha}(x)^2} \, dx.
\]
Hence, for $n\not \in F_1$ and $\alpha \not \in\ZZ$,  we get
\begin{multline*}
   \frac{\Gamma(n+\alpha+1)\prod_{f\in F_1}(n -f)
   \prod_{f\in F_2}(n+\alpha+f+1)}{n!}
   \\
   = \int_0^\infty L_{n+u_\F}^{\alpha;\F}(x)^2
   \frac{x^{\alpha +k}e^{-x}}{\Omega_{\F}^{\alpha}(x)^2} \, dx
   \ge 0.
\end{multline*}
The admissibility condition~\eqref{defadmi} for $\alpha+1$ and $\F$ follows now easily.

If $\alpha\in \ZZ$, the theorem follows using an argument of continuity.
\end{proof}

\section{Appendix: Describing admissible pairs  $(c,\F)$}
We start this appendix by recalling that for exceptional Hermite polynomials, the admissibility of a finite set $F$ of positive integers is defined by (see \cite{duch}, \cite{GUGM})
\begin{equation}\label{aH}
\prod_{f\in F}(n-f)\ge 0, \quad n\in \NN.
\end{equation}
This concept of admissibility is easier than the one defined in~\eqref{defadmi} because of two reasons. On one hand, we have now a single finite set $F$ instead of a pair $\F$ of finite sets. On the other hand, the Hermite admissibility only depends on the finite set $F$ while Laguerre admissibility also depends on the parameter $\alpha$ of the Laguerre polynomials. Hermite admissibility~\eqref{aH} can be characterized easily: the maximal segments of $F$ have an even number of elements. More precisely, split up the set $F$, $F=\bigcup _{i=1}^KY_i$, in such a way that $Y_i\cap Y_j=\emptyset $, $i\not =j$, the elements of each $Y_i$ are consecutive integers and $1+\max (Y_i)<\min Y_{i+1}$, $i=1,\dots, K-1$ ($Y_i$, $i=1,\dots, K$, are called the maximal segments of $F$). Then $F$ satisfies~\eqref{aH} if and only if each $Y_i$, $i=1,\dots, K$, has an even number of elements. This characterization allows an easy generation of Hermite admissible sets $F$.

The purpose of this appendix is to find a similar characterization for real numbers $c$ and pairs $\F=(F_1,F_2)$ of positive integers satisfying the Laguerre admissibility defined by~\eqref{defadmi}.

Hence, consider a pair of finite sets of positive integers $\F = (F_1, F_2)$ and a real number $c \in \RR \setminus \{0,-1,-2,\dots\}$. The pair $(c,\F)$ is admissible, see~\eqref{defadmi}, if
\begin{equation}
\label{def}
   \frac{\prod_{f\in F_1} (n-f)
   \prod_{f\in F_2} (n+c+f)}
   {(n+c)_{\hat{c}}}\geq 0,
   \quad \forall n \in \NN ,
\end{equation}
where $\hat{c} = \max\{-[c],0\}$.

In case $c \geq 0$, condition~\eqref{def} reduces to the Hermite admissibility (see \eqref{aH}) for the set $F_1$,
\[
   \prod_{f\in F_1} (n-f) \geq 0
   \quad \forall n \in \NN .
\]
Let us assume now that $c < 0$. Then $\hat{c} = -[c]$ and~\eqref{def} becomes
\[
   \prod_{f\in F_1} (n-f)
   \prod_{f\in F_2} (n+c+f)
   \prod_{0\leq m < -[c]} (n+c+m) \geq 0
   \quad \forall n \in \NN .
\]
The terms in the second and third products are never zero if $c \notin \{0,-1,-2,\dots\}$. Now let us observe that, in the second product, those terms with $c + f > 0$ can obviuously be omitted. And those terms with $c + f < 0$ (or, equivalently, $[c]+f < 0$) are present also in the third product, so they can be omitted in both products. Therefore, the admissibility condition is equivalent to
\[
   \prod_{f\in F_1} (n-f)
   \prod_{\substack{0\leq m < -[c]\\ m \notin F_2}} (n+c+m) \geq 0
   \quad \forall n \in \NN .
\]
In other words:
\begin{equation}
\label{admisG}
   \prod_{f\in G} (n-f) \geq 0
   \quad \forall n \in \NN ,
\end{equation}
where $G = F_1 \cup \{-c-m;\ m \in \{0,1,\dots,-[c]-1\} \setminus F_2 \}$. Now  this condition can be expressed in terms of the set $G$ as follows. Let
\[
   \mathcal{S} = \NN 
   \cup \{-c-m;\ m \in \{0,1,\dots,-[c]-1\} \setminus F_2 \}.
\]
With the natural order, each element in $\mathcal{S}$ has a next element, and each one other than~$0$ has a previous one. A subset of $\mathcal{S}$ formed by consecutive elements can be called a segment. Any subset of $\mathcal{S}$ can be uniquely expressed as the union of maximal segments. Thus, condition~\eqref{admisG} holds if and only if these maximal segments have an even number of elements.

This characterization allows an easy generation of examples of admissible numbers $c$ and pairs $\F$.
For instance, take $c = -17/4$ and $\F=(F_1,F_2)$, with $F_1 = \{1,2,8,9\}$ and $F_2 = \{1,2\}$. Then,
\begin{align*}
   \mathcal{S} &= \{0,\frac{1}{4}, 1, \frac{5}{4}, 2, 3, 4, \frac{17}{4},
   5, 6, 7, \dots \},
   \\
   G &= \{\frac{1}{4}, 1, \frac{5}{4}, 2, \frac{17}{4}, 8, 9\},
\end{align*}
and the maximal segments in $G$ are $\{\frac{1}{4}, 1, \frac{5}{4}, 2\}$, $\{\frac{17}{4}\}$, and $\{8,9\}$. Since one of these segments  has an odd number of elements, the pair $(c,\F)$ is not admissible. From this example, we can however find many admissible pairs $(c,\F)$. Indeed, with the same choices for $c$ and $F_2$, take $F_1 = \{1,2,5,8,9\}$, the maximal segments in $G$ become $\{\frac{1}{4}, 1, \frac{5}{4}, 2\}$, $\{\frac{17}{4},5\}$, and $\{8,9\}$, so the pair $(c,\F)$ is admissible. Similarly, we can see that for $F_1 = \{1,2,4,8,9\}$, the pair $(c,\F)$ is admissible as well.

\end{document}